\documentclass[a4paper,twoside,11pt]{amsart}

\usepackage{graphicx}
\usepackage[english]{babel}
\usepackage[T1]{fontenc}
\usepackage[ansinew]{inputenc}

\usepackage{amsmath,enumerate,color,nccmath,geometry,bm}
 \numberwithin{equation}{section}
\usepackage{amsthm}
\usepackage{amsfonts}
\usepackage{amssymb}
\usepackage{amsbsy}
\usepackage{graphics}
\usepackage{stmaryrd}
\usepackage{amsfonts}
\usepackage{setspace}
\usepackage{mathtools}

\newtheorem{theorem}{Theorem}[section]

\newtheorem{lemma}[theorem]{Lemma}

\theoremstyle{definition}
\newtheorem{definition}[theorem]{Definition}
\newtheorem{example}[theorem]{Example}

\theoremstyle{remark}
\newtheorem{remark}[theorem]{Remark}
\numberwithin{equation}{section}

\newcommand{\M}{\mathcal{M}}
\newcommand{\C}{\mathbb{C}}
\newcommand{\A}{\mathcal{A}}
\newcommand{\B}{\mathcal{B}}

\newcommand{\N}{\mathbb{N}}   
\newcommand\D{\mathbb D} 
\newcommand\diag{\operatorname{diag}}
\newcommand\eps{\varepsilon}
\renewcommand{\Re}{\text{\normalfont Re}} 
\renewcommand{\Im}{\text{\normalfont Im}} 

\begin{document}
\parindent 0pt 
\definecolor{red}{rgb}{0.9, 0.0, 0.0}
\definecolor{magenta}{rgb}{0, 0.9, 0.0}
\pagestyle{empty} 

\setcounter{page}{1}

\pagestyle{plain}

\title{A quantum remark on biholomorphic mappings on the unit ball}
\author{Sebastian Schlei{\ss}inger}
\thanks{Supported by
the German Research Foundation (DFG), project no. 401281084.}

\date{\today}

\begin{abstract}
In this note we regard non-commutative probability theory with operator-valued expectation. 
We show that the moment generating functions of distributions coming from 
monotone increment processes of unitary random variables yield biholomorphic mappings on
certain higher dimensional unit balls. 
\end{abstract}

\maketitle

\tableofcontents

\newpage

\section{Introduction}

In 1923, Loewner introduced a differential equation for conformal mappings to attack the Bieberbach conjecture:\\
 Let $\D=\{z\in\C\,|\, |z|<1\}$ be the unit disc and let $f:\D\to \C$ be a univalent (=holomorphic and injective)
 function normalized with 
$f(z)=z+\sum_{n\geq 2}a_nz^n$. The Bieberbach conjecture states that 
$ |a_n|\leq n$ for all $n\geq 2$. Loewner could prove the case $n=3$ (\cite{MR1512136}) 
and since then,
his approach has been extended and the Loewner differential equations are now an important tool in the 
theory of conformal mappings, in particular after the invention of SLE, the Schramm-Loewner evolution. His equations have been used also in the final proof of the Bieberbach conjecture 
by de Branges in 1985.\\

The Riemann mapping theorem -- the very foundation of geometric function theory in one dimension -- 
has no comparable counterpart in higher dimensions. Nevertheless, one can study biholomorphic mappings 
and Loewner theory also on domains in $\C^n$, or even in complex Banach spaces.\\
The most studied subdomains of $\C^n$ are the Euclidean unit ball and the polydisc $\D^n$. In this note we show that 
biholomorphic mappings and Loewner chains on certain unit balls, and thus 
also geometric function theory, appear naturally in quantum probability theory. \\
Indeed, this connection has been pointed out already for dimension one in \cite{Bau03}, \cite{bauer04}, \cite{monotone}, and \cite{FHS18}. In \cite{Jek17}, Jekel regards the higher dimensional case on upper half-spaces (the so called ``chordal'' case of the Loewner equation, 
or ``additive'' case in 
quantum probability). \\
In this paper we consider the ``radial'' case of the Loewner equation (or the ``multiplicative'' case  in quantum probability).\\

In Sections 2 and 3 we review some notions of higher dimensional Loewner theory and quantum probability theory. We show (Theorem \ref{thm:2})  that 
 distributions of certain unitary random variables embedded into monotone increment processes 
 have moment generating functions which are biholomorphic.

\section{Normalized Loewner chains on unit balls in $\C^n$}

We fix a norm $\|\cdot\|$ on $\C^n$ and let $B=\{z\in \C^n\,|\, \|z\|<1\}$ be the corresponding unit ball.
We denote by $I$ the identity matrix on $\C^n$.

\begin{definition}A (decreasing) subordination chain is a family $(f_t)_{t\geq0}$ of holomorphic functions 
$f_t:B\to \C^n$ such that $f_t = f_s\circ f_{s,t}$ for some holomorphic $f_{s,t}:B\to \C^n$ whenever $s\leq t$. \\
 A (normalized decreasing) \emph{Loewner chain} 
 on $B$ is a family $(f_t)_{t\geq0}$ of biholomorphic mappings $f_t:B\to f_t(B)$ such that 
 $f_t(B) \subset f_s(B)$ whenever $s\leq t$, $f_0(z)=z$ for all $z\in B$, and $f_t(0)=0$, $Df_t(0)=e^{-t}I$ 
 for all $t\geq 0$.
\end{definition}

Obviously, every Loewner chain is a subordination chain as $f_{s,t}$ can be defined by $f_{s,t}=f_s^{-1}\circ f_t$.
The literature usually focuses on normalized increasing Loewner chains, which simply means that we have
$f_s(B) \subset f_t(B)$ whenever $s\leq t$ and $Df_t(0)=e^tI$. 
Clearly, if $(f_t)_{t\geq0}$ is a normalized decreasing Loewner chain, then $(e^Tf_{T-t})_{0\leq t\leq T}$ 
is a normalized increasing Loewner chain for any $T>0$. \\
Normalized Loewner chains on the unit ball are intensively investigated in the literature, see the book \cite{GK03}.\\

The following class is important to set up a differential equation for normalized Loewner chains.

\begin{definition}
The class $\M(B)$ of all \emph{Herglotz functions} on $B$ 
is defined as the set of all holomorphic $h:B\to \C^n$ with $h(0)=0$, $Dh(0)=-I$
and \[ \Re(l_z(h(z))) <0 \]
for all $z\in B\setminus\{0\}$ and all linear functionals $l_z:\C^n\to\C$ with $l_z(z)=\|z\|, \|l_z\|=1$.\\
A \emph{Herglotz vector field} is a function $M:B\times [0,\infty)\to\C^n$ such that $M(\cdot,t)\in \M$ for all 
$t\geq0$ and $t\mapsto M(z,t)$ is measurable for all $z\in B$.
\end{definition}

\begin{example}
 In case of the Euclidean norm, the condition $\Re(l_z(h(z))) <0 $ becomes 
 $\Re\left<h(z),z\right><0$ for all $z\in B\setminus\{0\}$.
\end{example}

\begin{example}  \label{herglotz}
For $n=1$, every $h\in \M(\D)$ has the form $h(z)=-zp(z)$ for a holomorphic function 
$p\in \mathcal{P}(\D)$, where $\mathcal{P}(\D)$ denotes the Carath\'{e}odory class, i.e.\ the set of
 all holomorphic functions $p:\D\to \C$ with $\Re(p(z))>0$ for all  $z\in \D$ and $p(0)=1$.
The class $\mathcal{P}(\D)$ can be characterized by the Riesz-Herglotz representation formula:
\begin{equation*}   \mathcal{P}(\D) = \left\{
\int_{\partial \D} \frac{u + z}{u -z} \, \mu(du) \,|\, \mu \; 
\text{is a probability measure on $\partial \D$} \right\}.  \end{equation*}
The extreme points of the class $\mathcal{P}$ are thus given by all functions of the form
$\frac{u+z}{u-z}$ for some $u\in \partial \D$. 
\end{example} 

An important property of $\M$ is compactness.

\begin{theorem}[Theorem 6.1.39 in \cite{GK03}]
 The class $\M$ is compact with respect to locally uniform convergence.
\end{theorem}

\begin{theorem}\label{thm:0}Let $(f_t)_{t\geq0}$ be a subordination chain on $B$ satisfying $f_0(z)=z$ for all $z\in 
B$, $f_t(0)=0$, and 
$Df_t(0)=e^{-t}I$ for all $t\geq0$. 
Then $(f_t)_{t\geq0}$ is a Loewner chain and it satisfies the following partial 
differential equation:
 \begin{equation}\label{EV_Loewner}
\frac{\partial}{\partial t} f_{t}(z) = 
Df_{t}(z)\cdot M(z,t) \quad \text{for a.e. $t\geq 0$, $f_{0}(z)=z\in B,$}
\end{equation}
where $M$ is a Herglotz vector field.\\
Conversely, if $M$ is a Herglotz vector field, then there exists exactly one solution $f_t(z)$
 to \eqref{EV_Loewner} such that $t\mapsto f_t(z)$ is absolutely 
continuous for every $z\in B$, and $(f_t)_{t\geq 0}$ is a decreasing Loewner chain.
\end{theorem}
\begin{proof}
 Let $(f_t)_{t\geq0}$ be a subordination chain on $B$ with the given normalization and let $(f_{s,t})_{0\leq s\leq t}$ 
 be the transition mappings. We have to show that every $f_t$ is biholomorphic. \\
 From $f_t = f_s\circ f_{s,t}$ we see that $Df_{s,t}(0)=e^{s-t}I$. Furthermore, 
 $f_s$ is invertible in a neighborhood of $0$ and the identity principle implies that $f_{s,t}$ is uniquely determined. In particular, 
 $f_{s,s}(z)=z$ for all $z\in B$ and $s\geq0$.\\
 Now we fix some $T>0$.  We consider the family $(v_{s,t})_{0\leq s\leq t\leq T} := (f_{(T-t),(T-s)})_{0\leq s\leq t\leq T}$.
  Then $v_{s,t}(0)=0$, $Dv_{s,t}(0)=e^{s-t}I$, $v_{s,s}(z)=z$ and $v_{s,u} = v_{t,u}\circ v_{s,t}$ for all 
  $0\leq s\leq t\leq u\leq T$. \\
  We can now follow the proof of \cite[Theorem 2.2]{GKK03} to see that there exists a Herglotz 
  vector field 
  $H(z,t)$ such that 
  \begin{equation}\label{loewner_inv} \frac{\partial v_{s,t}(z)}{\partial t} = H(v_{s,t}(z), t)\quad 
\text{for a.e. $s\leq t\leq T$, $v_{s,s}(z)=z$}. \end{equation}

 We point out some of the crucial steps: 
   
 Define $g_{s,t}(z) = \frac{v_{s,t}(z)-z}{1-e^{s-t}}$. Then $g_{s,t}\in \M$ and the compactness of $\M$ implies that 
 $\|g_{s,t}(z)\|\leq M(r)$ for all $\|z\|\leq r$ and some $M(r)>0$. Hence,
 \[\|v_{s,t}(z)-v_{s,u}(z)\| = \|v_{s,t}(z)-v_{t,u}(v_{s,t}(z))\|\leq M(r)(1-e^{t-u})\leq M(r)(u-t).\]
Thus, for $z\in B$ fixed, we have that $t\mapsto v_{s,t}(z)$ is Lipschitz continuous and thus differentiable 
almost everywhere. An application of Vitali's theorem shows that $t\mapsto v_{s,t}(z)$ 
is differentiable for all $z$ and almost every $t \in [s,T]$. In case of differentiability, 
the Herglotz function $H(z,t)$ can be obtained 
by the limit $u\to t$, $u\geq t$, of the difference quotient 

 \begin{equation}\label{diff_quotient}
 (v_{t,u}(z)-z)/(u-t)=(1-e^{t-u})/(u-t)\cdot g_{t,u}(z)\to H(z,t).
 \end{equation}
 Then we have
 \[\frac{\partial}{\partial t} v_{s,t}(z)=\lim_{u\to t}(v_{s,u}(z)-v_{s,t}(z))/(u-t)=
 \lim_{u\to t}(1-e^{t-u})/(u-t)\cdot g_{t,u}(v_{s,t}(z))=H(v_{s,t}(z),t).\]

  The solution of \eqref{loewner_inv} is unique and produces univalent functions, see \cite[Lemma 1.3]{GKK03}. Hence, $(f_t)_{t\geq0}$ is a Loewner chain. 
  Equation \eqref{EV_Loewner} follows now from the relation 
  $f_{0,t}\circ v_{0,T-t} = v_{T-t, T}\circ v_{0,T-t}=v_{0,T}$ with $M(z,t)=H(z,T-t)$.\\

  Conversely, if a Herglotz vector field $M(z,t)$ is given, we can fix some 
  $T>0$ and define the Herglotz vector field $H(z,t)=M(z,T-t)$. 
  \cite[Lemma 1.3]{GKK03} implies that there is a unique solution $z\mapsto v_{s,t}(z)$ of \eqref{loewner_inv} and that 
  $v_{s,t}:B\to v_{s,t}(B)\subset B$ is a biholomorphic mapping. It follows that 
  $(f_{t})_{0\leq t\leq T}:=(v_{T-t, T})_{0\leq t\leq T}$ is a part of a decreasing Loewner chain satisfying \eqref{EV_Loewner}.\\
	
	Let us choose another $S>T$, to obtain a family $(\hat{f}_{t})_{0\leq t\leq S}:=(\hat{v}_{S-t, S})_{0\leq t\leq S}$ in the same way.
	Let $s\in [S-T, S]$. Then  $\hat{v}_{s,S} = v_{s+T-S,S+T-S}=v_{s+T-S,T}$. Hence 
$\hat{f}_t = \hat{v}_{S-t,S} = v_{T-t,T}=f_t$ for all $t \in [0, T]$. The two Loewner chains coincide on $[0,T]$ and  as we can choose $T>0$ arbitrarily large, we conclude that there exits a decreasing Loewner chain $(f_t)_{t\geq 0}$ satisfying \eqref{EV_Loewner}.\\

	Let $(h_{t})$ be another solution of \eqref{EV_Loewner}, locally absolutely continuous in $t$ for every $z\in B$. 
	
	Let $T>0$ and $(f_{t})_{0\leq t\leq T}=(v_{T-t, T})_{0\leq t\leq T}$ as defined above. Then
\begin{eqnarray*}
&&\frac{\partial}{\partial t}[(h_t\circ v_{0,T-t})(z)]=\\
 &&  Dh_{t}(v_{0,T-t}(z))\cdot M(v_{0,T-t}(z),t)  - Dh_t(v_{0,T-t}(z))\cdot M(v_{0,T-t}(z),t) = 0
\end{eqnarray*} 
for a.e.\ $t\leq T$ and thus $h_t(v_{0,T-t}(z))=h_0(v_{0,T}(z))=v_{0,T}(z)$ for all $z\in B$ and $0\leq t\leq T$. This implies $h_t = v_{0,T} \circ v^{-1}_{0,T-t}=v_{T-t,T}=f_t$ on $v_{0,T-t}(B)$. As $v_{0,T-t}(B)$ is an open set, the identity theorem implies 
$h_t = v_{T-t,T}=f_t$ on $B$. We can choose $T$ arbitrarily large, so $h_t=f_t$ for all $t\geq0$.
	
  
\end{proof}
%

\section{Quantum probability}

Let $\A$ be a unital $C^*$-algebra. An element $a\in \A$ is called \emph{self-adjoint} if $a^*=a$ and it is called \emph{unitary}
if $aa^*=a^*a=1$.
For $a\in \A$, the real and imaginary parts are defined by $\Re(a)=(a+a^*)/2$ and $\Im(a)=(a-a^*)/(2i)$. Note that both are self-adjoint.\\
Finally, if $a\in\A$ is self-adjoint with spectrum contained in $[0,\infty)$, 
we write $a\geq 0$ and $a$ is called positive. This is equivalent to the existence 
of $b\in \A$ with $a=b^*b$. \\
We call a self-adjoint element $a$ strictly positive, in short $a>0$, if there exists $\eps>0$ such that $a -\eps\cdot 1\geq 0$.
Then $a$ is invertible as its spectrum is contained in $[\eps,\infty)$.

\begin{example}
 Equip $\C^n$ with the Euclidean inner product and $\A=\C^{n\times n}$ with the corresponding 
 operator norm. By defining $a^*=\overline{a^T}$, $\A$ becomes a unital $C^*$-algebra.
\end{example}

We collect some simple properties of unital $C^*$-algebras.

\begin{lemma}${}$\label{calc}Let $a\in \A$.
\begin{itemize}
	\item[(a)] If $\Re(a)>0$\footnote{Note that, in general, $\Re(a)>0$ 
	is not equivalent to ``the spectrum of $a$ lies in the right half-plane''. For $a=\begin{pmatrix} \lambda & 1\\0 & \lambda\end{pmatrix}$ we have
the eigenvalue 
$\lambda$, while $\Re(a)=\begin{pmatrix} \Re(\lambda) & 1/2\\1/2 & \Re(\lambda)\end{pmatrix}$ 
has eigenvalues 
$\Re(\lambda)\pm\frac1{2}$.
}, then $a$ is invertible.
	\item[(b)] We have $\|a\|<1$ if and only if $1-aa^*>0$. 
	\item[(c)] Let $a,b\in\A$. If $a\geq 0$, then $bab^*\geq 0$.
	\item[(d)] Assume $\Re(a)> 0$ and put $w=(1-a)(1+a)^{-1}$. Then 
	$\|w\|<1$. (Note that $1+a$ is invertible by (a).) 	
		\item[(e)] Let $w\in \A$ with $\|w\|<1$ and put $b=(1-w)(1+w)^{-1}$. Then 
	$\Re(b)>0$. 	
\end{itemize}
\end{lemma}
\begin{proof}${}$
\begin{itemize}
	\item[(a)] This follows from a similar calculation as in \cite[p. 100]{BPV15}.
	Let $x=\Re(a)>0, y=\Im(a)$. Then 
	\[a=x+iy = (\sqrt{x})^2+iy = \sqrt{x} (1 + i  \sqrt{x}^{-1} y \sqrt{x}^{-1} ) \sqrt{x}=
	-i\sqrt{x} (i -  \sqrt{x}^{-1} y \sqrt{x}^{-1} ) \sqrt{x}.\] 
	As $\sqrt{x}^{-1} y \sqrt{x}^{-1}$ is self-adjoint, $i$ does not belong to its spectrum and thus 
	$i -  \sqrt{x}^{-1} y \sqrt{x}^{-1}$ is invertible. So $a$ is invertible.
	

	
	\item[(b)] Assume that $1-aa^*\geq \eps 1$. Then  $0\leq aa^* \leq (1-\eps)1$ and thus $\|aa^*\|\leq (1-\eps)<1$ (if $b$ is self-adjoint and $-c1\leq b\leq c1$, then $\|b\|\leq c$). 
	The $C^*$-property implies $\|a\|^2=\|aa^*\|<1$.\\
Conversely, assume that $\|a\| = 1-\eps < 1$. Then  $aa^*\leq \|aa^*\|1 = \|a\|^2 1 < (1-\eps)^21=:(1-\delta)1$. So $(1-\delta)1-aa^*\geq 0$, which shows $1-aa^*>0$.

\item[(c)] We have $bab^* = b\sqrt{a}\sqrt{a}b^* = (b\sqrt{a})(b\sqrt{a})^*\geq 0$.

\item[(d)] Let $\eps>0$ so small that $\eps(1+\|aa^*\|)1 \leq (4-2\eps)\Re(a)$. As $\eps(1+aa^*)\leq \eps(1+\|aa^*\|)1$, we have 

$(4-2\eps)\Re(a) - \eps(1+aa^*) \geq 0$, or $4\Re(a) -\eps(1+a)(1+a^*) \geq 0$. 
Furthermore, $4\Re(a)=2(a+a^*)=(1+a)b(1+a^*)$ with $b:=1-ww^*$. So $(1+a)(b-\eps 1)(1+a^*)\geq 0$.

Multiplying with $(1+a)^{-1}$ from the left and with $(1+a^*)^{-1}$ from the right yields with (c) that 
$b-\eps 1 \geq 0$ and (b) implies $\|w\|<1$.

\item[(e)] 

Let $\eps>0$ be so small that $2(1-ww^*) - \eps(1+w)(1+w^*)\geq 0$. 
As $(1+w)2\Re(b)(1+w^*)= (1-w)(1+w^*)+(1+w)(1-w^*)=2(1-ww^*)$, we have 
$(1+w)(2\Re(b)-\eps 1)(1+w^*)\geq 0$ and thus, with (c), also $2\Re(b)-\eps 1\geq 0$.
\end{itemize}
\end{proof}

Also $\A^{n\times n}$ can be regarded as a $C^*$-algebra (see \cite[p. 1-8]{Lan95}): We define $(a_{jk})^*=(a_{kj}^*)$. 
Furthermore, for $x=(x_1,...,x_n), y=(y_1,...,y_n)\in \A^n$, 
define the $\A$-valued inner product as 
$(x,y)=x_1^*y_1+...+x_n^*y_n$, and $\|x\|=\sqrt{(x,x)}$. For $a\in \A^{n\times n}$ we let 
$\|a\|=\sup\{\|(a(y),x)\|\,|\, \|x\|\leq 1, \|y\|\leq 1\}$, where 
$a(y)=(c_j), c_j=\sum_{k}a_{jk}y_k$. \\

\begin{definition}Let $\A, \B$ be both unital $C^*$-algebras with $\B\subset \A$. We call a linear map $\Phi:\A\to\B$ a 
\emph{$\B$-valued expectation} if
\begin{itemize}
\item[(a)]$\Phi(1)=1$, 
\item[(b)] $\Phi(bab')=b\Phi(a)b'$ for all $b,b'\in\B$ and $a\in \A$, and
\item[(c)] $\Phi$ is \emph{completely positive}: $ \Phi(A^*A) \geq 0$ for all $A\in \mathcal{A}^{n\times n}$ and all $n\in\N$, where we extend $\Phi$   componentwise to matrices.
\end{itemize}
 The triple $(\A,\B,\Phi)$ is called a 
($\B-$valued) \emph{non-commutative probability space} or \emph{quantum probability space}. 
\end{definition}

We note that the positivity of $\Phi$ implies $\Phi(a^*)=\Phi(a)^*$.\footnote{Write $a=x+iy$ with 
$x=\Re(a), y=\Im(a)$. Then $\Phi(a)=\Phi(x)+i\Phi(y)$. 
Furthermore, decompose $x$ and $y$ into their positive and negative parts, 
$x=x_+-x_-, y=y_+-y_-$. Then $\Phi(x_+), \Phi(x_-), \Phi(y_+), \Phi(y_-)$ are positive elements
and we conclude that $\Phi(x)$ and $\Phi(y)$ are self-adjoint, which implies 
$\Phi(a)^*=\Phi(x)-i\Phi(y)=\Phi(x-iy)=\Phi(a^*)$.}
%

\begin{example}
 An important example of a quantum probability space $(\A,\B,\Phi)$ is given by 
 the one-dimensional case $\B=\C$, and $\A$ is the set of all bounded linear operators on a Hilbert space $H$
 with  
 \[\Phi:\A\to\mathbb{C}, \quad \Phi(A)=\langle\xi, A\xi\rangle,\]
 where $\xi\in H$ is a fixed unit vector. 
\end{example}


\begin{definition}
The distribution $\mu_a$ of a random variable $a\in \A$ is defined as the map 
$\mu_a: \B\left<X\right>\to \B$, $\mu_a(p)=\Phi(p(a))$, where $\B\left<X\right>$ 
denotes the set of all non-commutative 
polynomials over $\B.$ 
\end{definition}

The distribution of $a$ is uniquely determined by the set of all mixed moments 
$\Phi(b_1 a b_2 \cdots a b_{n})$, $n\in\N$.


\begin{example}If $\B=\C$, then the distribution $\mu$ of a random variable $U$ 
is uniquely determined by the moments $\Phi(U^n)$, $n\geq 1$. 
In case of a unitary random variable $U$, we can describe such a distribution also by the moment 
generating function
\begin{equation}\label{mom}
 \psi_{\mu}(z)= \sum_{n=1}^\infty \Phi((Uz)^{n})=\sum_{n=1}^\infty \Phi(U^n) z^{n},\quad z\in\D.
\end{equation}
 We can also write 
 \[\psi_{\mu}(z)  = \int_{\partial\D}\frac{xz}{1-xz}\,\alpha(dx)\]
 for a probability measure $\alpha$ on $\partial\D$, and $\mu$ is uniquely determined by $\alpha$ 
 and vice versa.
 \end{example}

In case of a general $C^*$-algebra $\B$, we can also define the moment generating function 
\eqref{mom}, which is then defined for all $z$
belonging to the unit ball of $\B$. However, this function does not contain all mixed moments 
necessary to recover the distribution. This issue is solved by allowing $z$ to be from the unit ball of $\B^{n\times n}$, 
 the set of all $n\times n$-matrices over $\B$, for each $n\in\N$. If $U\in\A$ and $z\in \B^{n\times n}$, 
then $Uz\in \A^{n\times n}$ and $\Phi(Uz)\in \B^{n\times n}$. 
We recover $\Phi(Ub_1\cdots U b_{n})$ from 
\[b=\begin{pmatrix}0& b_1& 0& ...& 0 \\ 0& 0& b_2& ...& 0\\ & \vdots &&&&\\
     0&0&0&...&b_{n}\\ 0&0&0&...&0   
\end{pmatrix} \quad \text{and}\quad  \Phi((Ub)^n)=\begin{pmatrix}0& ...&0& \Phi(Ub_1Ub_2\cdots Ub_{n}) \\
0& ...&0& 0\\  & \vdots &&\vdots\\
     0&...&0&0
    
\end{pmatrix}.\]



For unitary $U\in \A$ with distribution $\mu$ we define the moment generating function $\psi_\mu$ by 
\begin{equation}\label{def:psi}
 \psi_\mu(z):=\Phi\left(\frac{Uz}{1-Uz}\right) = \sum_{m=1}^\infty \Phi((Uz)^{m}),
\end{equation}
which is defined for all $z \in \mathbb{B}_n:=\{z\in \B^{n\times n}\,|\, \|z\|<1\}$ and all $n\geq1$.\\
We have \[w=\frac{Uz}{1-Uz}=\frac{1}{2}\frac{1+Uz}{1-Uz}-\frac1{2}.\]
By Lemma \ref{calc} (e) we have 
$\Re(w)+\frac1{2}> 0$. So there exists $\eps>0$ (dependent on $w$) such that $\Re(w)+\frac1{2}\geq \eps 1$. As $\Phi(\Re(w))=\Re(\Phi(w))$ and because $\Phi$ is completely positive, we also 
have $\Re(\Phi(w+\frac1{2}-\eps 1))=\Re(\Phi(w))+\Re(\Phi(\frac1{2}-\eps 1))=\Re(\Phi(w))+\frac1{2}-\eps 1\geq 0$. So $\Re(\Phi(w))+\frac1{2}>0$.\\
So $z\mapsto \psi_\mu(z)$ maps the unit ball $\mathbb{B}_n$ into the half-space 
$\Re(w)>-\frac1{2}$. 
By Lemma \ref{calc} (d), we can now go back to the unit ball by applying again a linear fractional transformation.
The $\eta$-transform is defined by
\begin{equation}\label{def:eta}
  \eta_\mu:\cup_{n\geq 1}\mathbb{B}_n\to\cup_{n \geq 1}\mathbb{B}_n, \quad \eta_\mu(z)=\frac{\psi_\mu(z)}{1+\psi_\mu(z)}.
\end{equation}
We denote by $\eta_\mu^{n}$ the restriction of $\eta_\mu$ to $\mathbb{B}_n$. 
Then $\eta_\mu^n$ maps $\mathbb{B}_n$ into itself.


\section{Loewner chains from monotone increment processes}\label{sec_4}

The notion of independence is of vital importance for classical probability theory. In a certain sense, there are only five suitable notions of 
independence in the non-commutative setting (for $\B=\C$):
tensor, Boolean, free, monotone and anti-monotone independence;  see \cite{MR2016316}.\\
We now look at monotone independence, whose discovery can be traced back to
the construction of a monotone Fock space by N. Muraki (\cite{MR1467953, MR1462227})
and De. Giosa, Lu (\cite{MR1483010, MR1455615}) from the years 1996 and 1997. Around 2000, 
Muraki abstracted the computation rule for mixed moments of creation and annihilation operators on the monotone
Fock space and arrived at the concept of monotone independence in \cite{M00,M01,Mur01b}.
The operator-valued case has been considered in \cite{hs14}, \cite{AW16}, and \cite{Jek17}.\\

Let $(\A,\B,\Phi)$ be a quantum probability space. Furthermore, assume that $\B$ is finite dimensional. 
Any finite dimensional $C^*$-algebra is isomorphic to a subalgebra of $\C^{m\times m}$ for some $m$ large enough. 
In this situation each $\eta_\mu^n$ is a holomorphic function:\\ 

Let $b_1,...,b_N$ be a basis of $\B^{n\times n}$.  
Then $z\in \B^{n\times n}$ can be written as $\sum_{l=1}^N z_l b_l$ with complex variables $z_l$. In $\Phi((Uz)^m)$, we can place the variables $z_k$ outside of the argument of $\Phi$ to see that $\Phi((Uz)^m)$ is a polynomial in $z_1,...,z_N$. Thus $\eta_\mu^n$ is holomorphic.\\

For $m=1$ we have $\Phi(Uz)=\Phi(U)z$ by property (b) of $\Phi$. Thus, if $\Phi(U)=c\cdot 1$, $c\in\C$, then 
the Jacobian $D\eta_\mu^n(0)$ of $\eta_\mu^n$ at $0$ is given by $cI$.

\begin{definition}
A subset $C\subseteq \A$ is called a $\B$-subalgebra if $C$ is a subalgebra of $\A$ 
and $\B\cdot C\cdot \B\subseteq C$.\\
For $X\in \A$, we denote by $\B\left<X\right>_0$ the $\B$-subalgebra 
consisting of finite sums of elements of the form 
$b_1Xb_2X...Xb_{n+1}$, $n\geq 1$, $b_1,....,b_{n+1}\in\B$. ($\B\left<X\right>_0$ contains no degree-zero terms.)
\end{definition}

\begin{definition}${}$\label{def-mon}
Let $Q=(\A,\B,\Phi)$ be a non-commutative probability space. A family of $\B$-subalgebras 
$(\mathcal{A}_\iota)_{\iota\in I}$ of $\A$ 
indexed by a linearly ordered set $I$
is called \emph{monotonically independent} (in $Q$) if the following condition is satisfied:\\
For any $n\in\N$, $j_1, ..., j_n \in I$ and any $X_1\in\mathcal{A}_{j_1},...,X_n\in\mathcal{A}_{j_n}$, we have
\[
\Phi(X_{1}\cdots X_n) = \Phi(X_1\cdots X_{p - 1}\Phi(X_p)X_{p + 1}\cdots X_n)
\]
whenever $p$ is such that $j_{p - 1} < j_p > j_{p + 1}$. (One of the inequalities is eliminated if $p = 1$ or $p = n$.)\\
For $X_1,...,X_n\in\A$, the tuple $(X_1,...,X_n)$ is called monotonically independent if the
algebras $\A_1=\B\left<X_1\right>_0,..., 
\A_n=\B\left<X_n\right>_0$ are monotonically independent.
\end{definition}

\begin{example}
 Let $X\in\A$. Then $(X,1)$ is always monotonically independent as 
 $\B\left<1\right>_0=\B$.\\
 Now assume that  $(1,X)$ is monotonically independent. Then 
 \[\Phi(b_1X\cdots Xb_n)=\Phi(b_1X\cdots Xb_{n-1}Xb_n)=\Phi(b_1X\cdots Xb_{n-1}\Phi(X)b_n)=
 \Phi(b_1X\cdots Xb_{n-1})\Phi(X)b_n,\] which implies 
 $\Phi(b_1X\cdots Xb_n)=b_1\Phi(X)\cdots \Phi(X)b_n$. Thus the distribution $\mu$  of $X$ is simply the 
 delta distribution at $\Phi(X)$.
\end{example}


\begin{theorem}\label{conv}
Let $U,V\in \A$ be unitary operators with distributions $\mu$ and $\nu$ such that $(U-1,V)$ is monotonically independent. 
Denote by $\alpha$ the distribution of $UV$. Then we have 
\[\eta_{\alpha}(z) = \eta_{\mu}(\eta_\nu(z))\]
for all $z\in \cup_{n\geq 1}\mathbb{B}_n$.
\end{theorem}
Hence the distribution $\alpha$ depends on $\mu$ and $\nu$ only and $\mu \rhd \nu := \alpha$ defines the \emph{multiplicative 
monotone convolution}. This convolution has been introduced for the case $\B=\C$ by Bercovici in 
\cite{B05}; see also \cite{franz06}.
\begin{proof}
The proof is analogous to the case $\B=\C$, see \cite[Theorem 4.1]{franz06}.\\ 
The statement is equivalent to $\psi_{\alpha}(z) = \psi_{\mu}(\eta_\nu(z))$ for all $z\in \mathbb{B}_n$ and all $n\geq 1$.\\
Let $X=U-1$ and write
 
  \begin{eqnarray*}
  &&\frac{UVz}{1-UVz}=\sum_{n=1}^\infty (UVz)^{n}=
  \sum_{n=1}^\infty (X+1)Vz...(X+1)Vz\\
  &=& \sum_{n=1}^\infty
  \sum_{k=1}^n \sum_{\substack{\nu_1+...+\nu_k=n\\ \nu_1,...,\nu_k\geq 1}} (X+1)(Vz)^{\nu_1}X(Vz)^{\nu_2}X...(Vz)^{\nu_k}.
 \end{eqnarray*}
 
Now we can use monotone independence:
 \begin{eqnarray*}
 &&\psi_\alpha(z)=\Phi\left(\frac{UVz}{1-UVz}\right)= 
    \sum_{n=1}^\infty\sum_{k=1}^n \sum_{\substack{\nu_1+...+\nu_k=n\\ \nu_1,...,\nu_k\geq 1}} 
    \Phi\left( (X+1)(Vz)^{\nu_1}X(Vz)^{\nu_2}X...(Vz)^{\nu_k}\right)\\
    &=&  \sum_{n=1}^\infty\sum_{k=1}^n \sum_{\substack{\nu_1+...+\nu_k=n\\ \nu_1,...,\nu_k\geq 1}} 
    \Phi\left((X+1)\Phi((Vz)^{\nu_1})X\Phi((Vz)^{\nu_2})X...\Phi((Vz)^{\nu_k})\right).
 \end{eqnarray*}
  Next, we have
  \begin{eqnarray*}
  && \frac{U\eta_\nu(z)}{1-U\eta_\nu(z)}=
  \left(U\frac{\psi_\nu(z)}{1+\psi_\nu(z)}\right)\left(1-U\frac{\psi_\nu(z)}{1+\psi_\nu(z)}\right)^{-1}\\
  &=& \left(U\psi_\nu(z)\right)\left(1+\psi_\nu(z)-U\psi_\nu(z)\right)^{-1}=
  (X+1)\psi_\nu(z)\left(1-X\psi_\nu(z)\right)^{-1}.
 \end{eqnarray*}
 Hence
   \begin{eqnarray*}
  &&  \psi_{\mu}(\eta_\nu(z))=\Phi\left(\frac{U\eta_\nu(z)}{1-U\eta_\nu(z)}\right)=
   \Phi\left((X+1) \sum_{n=1}^\infty 
  (\Phi((Vz)^{n})) \sum_{k=0}^\infty(X(\sum_{n=1}^\infty \Phi((Vz)^{n})))^{k}\right).
 \end{eqnarray*}
We see that this sum is indeed equal to $\psi_\alpha(z)$. Note that the expansion 
 $(1-X\psi_\nu(z))^{-1}=\sum_{k=0}^\infty (X\psi_\nu(z))^k$ only holds for 
 $\|z\|$ small enough, but the uniqueness of analytic continuation implies 
 that the two power series are indeed identical. 
\end{proof}
\begin{remark}\label{u_rm}
The monotone independence of $(U-1,V)$ is equivalent to the monotone independence of $(U-1,V-1)$.
\end{remark}

Next we introduce processes of unitary random variables with monotonically independent multiplicative increments. 
We refer to the books \cite{ABKL05} and \cite{MR2213451} for the general theory of classical and quantum processes 
with independent increments and quantum stochastic differential equations. 

%
\begin{definition}\label{def_umip}
Let $Q=(\A,\B,\Phi)$ be a non-commutative probability space and $(U_t)_{t\geq 0}\subseteq \A$ a
family of unitary random variables.  We  call $(U_t)_{t\geq 0}$ a
\emph{normalized unitary multiplicative monotone increment process (NUMIP)} if the following conditions are 
satisfied:
\begin{itemize}
\item[(a)] $U_0=1$ and $\Phi(U_t)=e^{-t}$ for all $t\geq 0$.
	\item[(b)] The tuple 
	\[(U_{t_1}^*U_{t_2}-1, ..., U_{t_{n-1}}^*U_{t_n}-1) \]
is monotonically independent for all $n\in\mathbb{N}$ and all $t_1,...,t_n\in\mathbb{R}$ with $0\leq t_1\leq t_2\leq ...\leq t_n$.
\end{itemize}

\end{definition}

\begin{theorem}\label{thm:2}
Let $\B$ be finite dimensional and let $(U_t)_{t\geq 0}\subseteq \A$ be a NUMIP with distributions $(\mu_t)_{t\geq 0}$. 
Then, for every $n\in\N$, $(\eta_{\mu_t}^n)_{t\geq0}$ is a normalized Loewner chain on 
$\mathbb{B}_n$ satisfying the differential equation \eqref{EV_Loewner}. \\
In particular, each $\eta_{\mu_t}^n:\mathbb{B}_n\to 
\eta_{\mu_t}^n(\mathbb{B}_n)$ is biholomorphic.
\end{theorem}
\begin{proof}
Denote by $\eta_{s,t}$ the $\eta$-transform of the distribution of $U_s^*U_t$, $s\leq t$. Let $0\leq s\leq t \leq u$. Then we have $ U_s^*U_u=   U_s^*U_tU_t^*U_u$. As $(U_s^*U_t-1, U_t^*U_u-1)$ is 
 monotonically independent, Theorem \ref{conv} and Remark \ref{u_rm} imply 
 \[ \eta_{s,u} = \eta_{s,t}\circ \eta_{t,u}.\]

 Fix $n\in\N$. Then $\eta^n_{0,0}(z)=z$ and the power series expansion of $\eta^n_{0,t}$ shows 
 $D\eta^n_{0,t}(0)=\Phi(U_t)I=e^{-t}I$, $\eta^n_{0,t}(0)=0$. Theorem \ref{thm:0} implies that $(\eta_{0,t}^n)_{t\geq 0}=(\eta_{\mu_t}^n)_{t\geq 0}$ is a Loewner chain.
\end{proof}

\begin{example}
Consider a quantum probability space $(\A,\C,\Phi)$ and let $h$ be the Haar measure on $\partial \D$.
Let $U\in\A$ be unitary with distribution $h$. 
The constant process $U_t\equiv U$ satisfies condition (b) from Definition \ref{def_umip}, but not (a).
We have $\psi_h(z)= \sum_{n=1}^\infty \Phi(U^n)z^n \equiv 0$ and thus $\eta_h(z)\equiv 0$ is not injective. 
\end{example}

\begin{remark}\label{inf_rem}(Infinite dimensions) 
The definition of NUMIPs is valid for general $C^*$-algebras and also Loewner chains can be defined 
on an infinite dimensional Banach space. If this Banach space is reflexive, then a Loewner chain also 
satisfies a Loewner differential equation, see \cite{GHKK13}. However, 
reflexive $C^*$-algebras are already finite-dimensional, see \cite[Theorem 6.3.6.15]{Con01}.\\
Reflexivity is needed in \cite{GHKK13} due to problems with absolute continuity and integrability, see 
\cite[Lemma 2.8]{GHKK13}. \\
In \cite{Jek17}, Loewner chains on non-reflexive Banach spaces are described via a distributional 
differential equation. It should be expected that Theorem \ref{thm:2} also holds in the general case.\\
For semigroups of univalent functions on $C^*$-algebras, see also \cite{EHRS02}.
 \end{remark} 
 
\section{Multivariate processes}

Fix a ``one-dimensional'' quantum probability space $(\A,\C,\Phi)$. 

We have seen that passing to an operator-valued expectation leads
to (bi)holomorphic mappings in several variables. This is also the case by regarding multivariate quantum processes in
$(\A,\C,\Phi)$.\\

Let $a,b\in \A$ be unitary. The distribution of the pair $(a,b)\in\A^2$ is defined by the set of all moments 
$\Phi(a^{j_1}b^{k_1}...a^{j_n}b^{k_n})$, $j_1,k_1,...,j_n,k_n\in\N\cup\{0\}$. This distribution can also be encoded via operator-valued expectation. 
Let $X=\diag(a,b)\in\A^{2\times 2}$. Then $X$ is a unitary random variable in 
the quantum probability space $(\A^{2\times 2}, \C^{2\times 2}, \Phi \otimes I_{\C^{2\times 2}})$. 
The distribution $\mu$ of $X$ is determined by the distribution of $(a,b)$ and vice versa.\\
Moreover, we can define monotone independence of $((a_1,b_1), (a_2,b_2))$ by the monotone independence 
of $(\diag(a_1, b_1),\diag(a_2,b_2))$. \\
 This is equivalent to the monotone independence of the $\C$-subalgebras 
$(\C\left<a_1,b_1\right>_0, \C\left<a_2,b_2\right>_0)$, where $\C\left<a,b\right>_0$ 
denotes the set of all finite sums of elements of the form $ca^{j_1}b^{k_1}...a^{j_n}b^{k_n}$, 
$c\in\C$, $n\geq 1$, $j_1,k_1,...,j_n,k_n\in\N\cup\{0\}$, and at least one exponent is greater than $0$.\\

Now assume that $a$ and $b$ commute. Then it suffices to consider $\eta_\mu$ only on 
$\mathbb{B}_1=\{z\in \C^{2\times 2}\,|\, \|z\|<1\}$. 
Even more is true. As observed in \cite[Section 3]{BBGS18} (in the case of self-adjoint operators), 
it is in fact sufficient to consider $\eta_\mu$ only on 
$\mathbb{B}^\Delta_1=\{z\in \mathbb{B}_1\,|\, \text{$z$ is upper triangular}\}$.\\ 
This can be seen as follows. It suffices to consider the moments $\Phi(a^jb^k)$, $j,k\geq0$.
Denote by $\alpha$ and $\beta$ the distribution of $a$ and $b$ within $(\A,\C,\Phi)$.
The distribution $\mu$ of $X$ is completely described by the following three functions
\[\psi_\alpha(z)=\Phi\left(\frac{az}{1-az}\right), \psi_\beta(w)= \Phi\left(\frac{bw}{1-bw}\right),
 \psi(z,w) = \Phi\left(\frac{a}{(1-az)(1-bw)}\right),\]
where $z,w\in\D$. Note that $\frac{a}{1-az}\cdot \frac{1}{1-bw}=\sum_{k\geq 0, l\geq 0}a^{k+1}b^lz^{k}w^l$.\\

Now consider $\psi_\mu(c)=\Phi(\frac{Xc}{1-Xc})$ for $c=\begin{pmatrix} z&\zeta\\0&w\end{pmatrix}\in\C^{2\times 2}$ 
with $\|c\|<1$. We have 
\[\frac{Xc}{1-Xc}=\begin{pmatrix} 
\frac{az}{1-az}&\zeta\frac{a}{(1-az)(1-bw)}\\0&\frac{bw}{1-bw}
\end{pmatrix} \quad \text{and thus} \quad 
\psi_\mu(c)=\Phi\left(\frac{Xc}{1-Xc}\right)=\begin{pmatrix} 
\psi_{\alpha}(z)&\zeta\psi(z,w)\\0&\psi_\beta(w)
\end{pmatrix}.\] This leads to
\[ (1+\psi_\mu(c))^{-1}=\begin{pmatrix} 
1+\psi_{\alpha}(z)&\zeta\psi(z,w)\\0&1+\psi_\beta(w)
\end{pmatrix}^{-1}= \begin{pmatrix} 
\frac1{1+\psi_{\alpha}(z)}&-\frac{\zeta\psi(z,w)}{(1+\psi_{\alpha}(z))(1+\psi_\beta(w))}\\0&\frac1{1+\psi_\beta(w)}
\end{pmatrix}\]
and
\[\eta_\mu(c) =
\begin{pmatrix} \eta_{\alpha}(z)&\zeta(1-\eta_\alpha(z))(1-\eta_\beta(w))\psi(z,w)\\0&\eta_\beta(w)\end{pmatrix}.\]
 
We see that $\eta_\mu$ maps $\mathbb{B}^\Delta_1$ into itself and that $\eta_\mu$ restricted 
to $\mathbb{B}^\Delta_1$ encodes the distribution $\mu$. Note that $\mathbb{B}^\Delta_1$ 
can be regarded as a unit ball in $\C^3$. With Theorem \ref{thm:2} we now obtain a Loewner chain 
on $\mathbb{B}^\Delta_1$ as follows.
 
\begin{theorem}\label{multi}Let $(a_t)_{t\geq0}$ and $(b_t)_{t\geq0}$ be two families of unitary random variables 
in $\A$ such that $a_t, b_t$ commute for every $t\geq0$. 
Let $X_t=\diag(a_t,b_t)$ with distribution $\mu_t$. Assume that $(X_t)_{t\geq 0}$ is a NUMIP in 
$(\A^{2\times 2}, \C^{2\times 2}, \Phi \otimes I_{\C^{2\times 2}})$. Then
$(\eta_{\mu_t})_{t\geq0}$ is a normalized Loewner chain on $\mathbb{B}^\Delta_1$. \\
In particular, each $\eta_{\mu_t}:\mathbb{B}^\Delta_1\to 
\eta_{\mu_t}(\mathbb{B}^\Delta_1)$ is biholomorphic.
\end{theorem}

\begin{remark}
 In the theorem above, we expressed the monotone independence of two pairs \\
 $((a_1,b_1), (a_2,b_2))$ by the monotone independence of $(\diag(a_1, b_1),\diag(a_2,b_2))$.\\
 However, there are further ways to define such an independence. 
 We refer to \cite{Ger17}, \cite{GHS17} and the references therein 
 for notions of monotone independence for pairs of random variables (bi-monotone independence). 
\end{remark}

\end{document}